\documentclass[12pt]{amsart}
\usepackage{amsmath, amssymb,color,url}
\usepackage[all]{xy}
\usepackage[margin=1in]{geometry}
\usepackage[colorinlistoftodos]{todonotes}

\newcommand{\mcN}{\mathcal{N}}
\newcommand{\mcP}{\mathcal{P}}
\newcommand{\mcZ}{\mathcal{Z}}
\newcommand{\NN}{\mathbb{N}}
\renewcommand{\sb}{\mathfrak{sb}}

\renewcommand{\ge}{\geqslant}
\renewcommand{\le}{\leqslant}

\theoremstyle{theorem}
\newtheorem{thm}{Theorem}
\newtheorem{lem}[thm]{Lemma}
\newtheorem{prop}[thm]{Proposition}
\newtheorem{cor}[thm]{Corollary}
\newtheorem{qn}[thm]{Question}

\theoremstyle{definition}
\newtheorem{defn}[thm]{Definition}
\newtheorem{ex}{Example}

\theoremstyle{remark}
\newtheorem{rem}[thm]{Remark}

\title{Global Fibonacci Nim}
\author{Urban Larsson}
\address{Department of Mathematics and Statistics, Dalhousie University, 6316 Coburg Road, PO Box 15000, Halifax, Nova Scotia, Canada, B3H 4R2}
\email{urban031@gmail.com}
\author{Simon Rubinstein-Salzedo}
\address{Euler Circle, Palo Alto, CA 94306}
\email{simon@eulercircle.com}
\date{\today}

\begin{document}
\maketitle

\begin{abstract}
\textsc{Fibonacci nim} is a popular impartial combinatorial game, usually played with a single pile of stones. The game is appealing due to its surprising connections with the Fibonacci numbers and the Zeckendorf representation. In this article, we investigate some properties of a variant played with multiple piles of stones, and solve the 2-pile case. A player chooses one of the piles and plays as in Fibonacci nim, but here the move-size restriction is a global parameter, valid for any pile. 
\end{abstract}

\section{Introduction}


The classical game of \textsc{Fibonacci nim}, as studied by Whinihan in~\cite{Whinihan63}, is played as follows: There is one pile of stones, with $n$ stones in the pile initially, and there are two players who take turns making moves. A move consists of removing some of the stones in the pile, subject to the following constraints: the first player must remove at least 1 stone, but may not remove the entire pile. On subsequent turns, if the previous player removed $m$ stones, then the next player must remove least one stone and at most $2m$ stones. The loser is the player who is unable to make a move (usually because there are no stones remaining, although there is also a special case in which the initial pile has one stone).

In his original paper on the game, Whinihan described the outcome of the game under optimal play:

\begin{thm}[Whinihan,~\cite{Whinihan63}] The first player has a winning strategy if and only if $n$ is not a Fibonacci number. \end{thm}

Furthermore, Whinihan gave a full winning strategy. This strategy relies on a celebrated theorem of Zeckendorf. However, it is also possible to give an alternative description of the winning strategy, in terms of partial sums of the so-called Fibonacci word. We introduce this word in~\S\ref{sec:Fibword} and deduce the winning strategy in terms of the Fibonacci word in~\S\ref{sec:Ppos}

It is natural to consider the game of \textsc{Fibonacci nim} played with more than one pile. In this game, one may remove stones from only one pile on any given move. However, there are two natural possibilities for the bound on the number of stones that may be removed: \begin{itemize} \item \textit{Local move dynamic:} Each pile has a separate counter, so that if the last move (by either player) in a pile was to remove $m$ stones, then the next move \emph{in that pile} must be to remove at most $2m$ stones. \item \textit{Global move dynamic:} There is only one counter for the entire game, so that if the previous move was to remove $m$ stones in \emph{any} pile, then the next move must be to remove at most $2m$ stones in \emph{any} pile (either the same pile, or a different pile). \end{itemize} In either case, it is natural to remove the restriction that the first player may not remove an entire pile; this artificial rule is necessary to make the one-pile game nontrivial, but it serves no further purpose in either multi-pile game.

The local move dynamic game is more natural from the perspective of combinatorial game theory, as the game is the \emph{disjunctive sum} of the individual piles. As a result, the game can be studied by means of the Sprague-Grundy theory (see~\cite{Grundy39,Sprague35}); the authors have previously analyzed this version in~\cite{LRS14}.

The global move dynamic game is probably the more natural one from the perspective of game play, and it must be analyzed differently, as the powerful tools based on Grundy values and disjunctive sums are not applicable. In~\S\ref{sec:twopile} we give the outcome class for all two-pile positions, first in terms of Zeckendorf representation, and then in terms of a generalized version of the Fibonacci word. In~\S\ref{sec:multipile} we study some properties of positions with several piles. Finally, in~\S\ref{sec:pow2nim}, we describe a simpler variant of the global move dynamic game, in which we can describe the full winning strategy.

We use the notation $(n_1,\ldots,n_k;r)$ to denote the global \textsc{Fibonacci nim} position with piles of size $n_1,\ldots,n_k$, where the maximum number of stones that can be removed on the first turn is $r$. We write $(n_1,\ldots,n_k;\infty)$ for the global \textsc{Fibonacci nim} position with piles of size $n_1,\ldots,n_k$, where any number of stones can be removed on the first turn, provided that they are all from the same pile.

\section*{Acknowledgments}

Part of the work for this paper was completed at the Games at Dal workshop at Dalhousie University in Halifax, Nova Scotia, in August 2015.

\section{$\mcN$ and $\mcP$ positions} \label{sec:NandP}

\begin{defn} We say that a game $G$ is an $\mcN$ position (resp.\ $\mcP$ position) and write $G\in\mcN$ (resp.\ $G\in\mcP$) if the player to move (resp.\ player not to move) has a winning strategy under optimal play. \end{defn}

Given an \emph{impartial} game (i.e.\ one in which both players have the same moves available to them, as opposed to e.g.\ \textsc{chess}, where one player moves the white pieces and one player moves the black pieces), there is a simple recursive characterization of the $\mcN$ and $\mcP$ positions.

\begin{prop} \label{prop:partition} $G\in\mcN$ if and only if there exists a move to a game $G'$ such that $G'\in\mcP$. \end{prop}

See~\cite[Theorem 2.13]{ANW07}.

Consequently, $G\in\mcP$ if and only if, for every move to a game $G'$, we have $G'\in\mcN$.

\section{Zeckendorf representation} \label{sec:zeckendorf}

A celebrated theorem of Lekkerkerker and Zeckendorf is the following:

\begin{thm}[Lekkerkerker~\cite{Lek52}, Zeckendorf~\cite{Zeckendorf72}] \label{thm:zeckendorf} Every positive integer $n$ can be expressed uniquely as a sum of pairwise nonconsecutive Fibonacci numbers with index at least 2. \end{thm}

\begin{defn} The Zeckendorf representation of $n$ is the unique sequence $z_1(n),z_2(n),\ldots,z_k(n)$ of Fibonacci numbers such that $z_1(n)+\cdots+z_k(n)=n$, and for all $1\le i<k$, $z_i(n)<z_{i+1}(n)$, and $z_i(n)$ and $z_{i+1}(n)$ are not consecutive elements of the Fibonacci sequence. We write $\mcZ(n)=\{z_1(n),\ldots,z_k(n)\}$. \end{defn}

For notational convenience, if $|\mcZ(n)|<k$, then we set $z_k(n)=\infty$, and we say that $z_k(n)>m$ for all integers $m$.


\section{The Fibonacci word} \label{sec:Fibword}

The Fibonacci word $W^{x,y} = f_0f_1f_2\cdots$ is a string of digits from some two-letter alphabet $\{x,y\}$. It is an archetype of a so-called Sturmian word; see~\cite{Loth02} for much more on Sturmian words. There are many equivalent ways of generating it.

\begin{prop}\label{prop:fibword} The following constructions give rise to the same sequence $f_0f_1f_2\cdots$: \begin{enumerate} \item Let $S_0=x$ and $S_1=xy$. For $n\ge 2$, let $S_n=S_{n-1}S_{n-2}$ be the concatenation of strings. Then, for every $n$, $S_n$ is an initial string of $S_{n+1}$. The Fibonacci word $S_\infty$ is the limiting string of the sequence $\{S_0,S_1,S_2,\ldots\}$. \item The Fibonacci word is the string $f_0f_1f_2\cdots$, where $f_n=x$ if $1\not\in \mcZ(n)$ and $f_n=y$ if $1\in\mcZ(n)$. \item The Fibonacci word is the unique non-trivial word $u$ on the alphabet $(x,y)$ where the parallel update $x\rightarrow yz, y\rightarrow y$, for all letters in $u$, gives back the same word, but now on the alphabet $(y,z)$. \end{enumerate} \end{prop}
See~\cite[p.\ 20]{Berstel86} or~\cite{Knuth97} for more details on the Fibonacci word, including a proof of Proposition~\ref{prop:fibword}, as well as other descriptions and interesting properties. Item (3) is obvious, by interpreting $y$ as $x$ and $z$ as $y$, which is the standard ``Fibonacci morphism". The beginning of the Fibonacci word is \[xyxxyxyxxyxxyxyxxyxyxxyxxyxyxxyxxy.\]
In the following, we make use of partial sums of the Fibonacci word, after substituting certain integers for $x$ and $y$. For instance, if we substitute $x=3$ and $y=7$, and among the first $m$ letters, there are $i$ $x$'s and $j=m+1-i$ $y$'s, then the partial sum is $3i+7j$. We sometimes write $W_i^{3,7}$, when we refer to the $i^\text{th}$ value $f_i\in \{3,7\}$ of $W^{3,7}$.
In view of item Proposition~\ref{prop:fibword} (3), a parallel update will often be interpreted as 
\begin{enumerate}
\item[(T1)] $F_i\rightarrow F_{i}$;
\item[(T2)] $F_{i+1}\rightarrow F_{i}F_{i-1}$,
\end{enumerate}
for some $i>1$.

We use the following lemmas on the sets of partial sums of Fibonacci words, which are easy consequences of part (3) of Proposition~\ref{prop:fibword}. These sets can also be described naturally in terms of the Zeckendorf representation, which would provide alternative proofs of some of our theorems in the rest of the article.

\begin{lem} \label{lem:partialsums} Suppose that $b\le a$. Then the set of partial sums of $w_a := W^{F_{a+1},F_a}$ 
is a subset of the set of partial sums of $W^{F_{b+1}, F_b}$. That is, let \[PS(w_a)=\left\{k:k=\sum_{i=0}^m W_i^{F_{a+1},F_a} \text{ for some } m\right\}.\] 
Then $PS(w_a)\subseteq PS(w_b)$.
\end{lem}

\begin{proof} It suffices to check this when $b=a-1$. In this case, any instance of $F_{a+1}$ in the Fibonacci word turns into $F_a,F_{a-1}$ after applying the morphism of part (3) of Proposition~\ref{prop:fibword} and making the substitutions. Any instance of $F_a$ remains as $F_a$. The result follows since $F_{a+1}=F_a+F_{a-1}$. \end{proof}

\begin{lem} \label{lem:lastPS} Suppose that $n\in PS(w_a)\setminus PS(w_{a+1})$. Then $n-F_{a+1}\in PS(w_{a+2})$. \end{lem}

\begin{proof} Suppose $n=\sum_{i=0}^m W_i^{F_{a+1},F_a}$, and we use the alphabet $(y,z)=(F_{a+1},F_a)$. We note that $n\in PS(w_a)\setminus PS(w_{a+1})$ if and only if $f_m=y$ and $f_{m+1}=z$, which follows immediately from part (3) of Proposition~\ref{prop:fibword}. Each instance of $yz$ is replaced by $x$, when using the reverse direction. This implies that $n-F_{a+1}\in PS(w_{a+1})$. Since $x$ represents $F_{a+2}$, the largest Fibonacci term in the word $w_{a+1}$, in going to $w_{a+2}$, the partial sum of all letters to the left of $x$ will remain a partial sum; this follows by applying the reverse of part (3) of Proposition~\ref{prop:fibword}.
\end{proof}



Observe that $PS(w_1)=\NN$ (since $F_1=F_2=1$), so that every nonnegative integer is in some $PS(w_a)$. Furthermore, $\bigcap_{a\ge 1} PS(w_a)=\{0\}$.

\section{$\mcP$ positions in one-pile \textsc{Fibonacci nim}} \label{sec:Ppos}

The winning strategy for one-pile \textsc{Fibonacci nim} was described by Whinihan. Consider the position $(n;r)$. If $z_1(n)\le r$, then $(n;r)$ is an $\mcN$ position, and removing $z_1(n)$ stones is a winning move. If $z_1(n)>r$, then $(n;r)$ is a $\mcP$ position, and there are no winning moves.

It is also possible to characterize the $\mcN$ and $\mcP$ positions in terms of the Fibonacci word. This approach will be useful for our analysis of the multi-pile game.

\begin{thm}\label{thm:alg} Fix a take-away size $r$. There is a unique Fibonacci number $F_t$ so that $F_t\le r<F_{t+1}$. The position $(n;r)\in\mcP$ if and only if $n$ is a partial sum of $W^{F_{t+1},F_{t}}$, i.e.\ if and only if there is some $m$ so that \begin{equation}\label{eq:fibnim}
n = \sum_{i=0}^m W^{F_{t+1},F_t}_i
\end{equation} 
\end{thm}

\begin{proof} Suppose first that $(n;r)$ is of the given form. We must demonstrate that there is no move to a position of the same form. Suppose that the new position is $(n-s;2s)$, with $s\le r < F_{t+1}$. Let $b$ be such that $F_b\le 2s < F_{b+1}$. Then $F_b\le 2s < 2F_{t+1}<F_{t+3}$, gives $b\le t+2$. We must prove that there is no $m$ such that 
\[n-s = \sum_{i=0}^m W^{F_{b+1},F_b}_i.\]
Since $s<F_b$ and $\min\{F_{b+1},F_b\} = F_{b}$, if $b\le t$, then (\ref{eq:fibnim}) together with Lemma~\ref{lem:partialsums} gives the claim. Suppose therefore that $b\in \{t+1, t+2\}$. If it were possible to play in $PS(w_{t+1})$ or $PS(w_{t+2})$, then, by definition of $t$, by (\ref{eq:fibnim}) and by the reverse of the (T2) composition (applied once or twice), the Fibonacci number $F_t$ has to be subtracted from $n$. Then, again by (T2), this gives the contradiction.

Suppose next that $(n;r)$ is not of the form in the statement of the theorem. Then there is an $m$ such that 
\begin{equation}\label{eq:fibnim2}
\sum_{i=0}^m W^{F_{t+1},F_t}_i < n < \sum_{i=0}^{m+1} W^{F_{t+1},F_t}_i.
\end{equation} 
There is a unique positive integer $b$ so that $n\in PS(w_b)\setminus PS(w_{b+1})$. By Lemma~\ref{lem:lastPS}, $n-F_{b+1}\in PS(w_{b+2})$.  Since $F_{b+2}\le 2F_{b+1}<F_{b+3}$, $(n-F_{b+1},2F_{b+1})$ is a $\mcP$-position.
\end{proof}

\section{Two-pile \textsc{Fibonacci nim}} \label{sec:twopile}
\subsection{The Zeckendorf approach}
The $\mcP$ positions of the two-pile \textsc{Fibonacci nim} game $(m,m+k;r)$ can also be expressed in terms of the Fibonacci word as a simple generalization of that of the one-pile game. Fix one pile size $m$ and the initial take-away amount $r$.

\begin{thm} \label{thm:twopiles} Let $t$ be such that $F_t\le r<F_{t+1}$. Then the following is a complete classification of the outcomes of the position $(m,m+k;r)$: \begin{enumerate} \item If $z_1(k)\le F_t$, then $(m,m+k;r)\in\mcN$. \item If $z_1(k)\ge F_{t+2}$, then $(m,m+k;r)\in\mcP$. \item If $z_1(k)=F_{t+1}$ and $m<F_t$, then $(m,m+k;r)\in\mcP$. \item If $m\ge F_t$ and $z_1(k)=F_{t+1}$, and either $z_2(k)=\infty$ or $z_2(k)=F_{t+d}$ where $m<F_t+F_{t+1}+\cdots+F_{t+d-3}$, then let $s$ be the unique integer so that $F_t+F_{t+1}+\cdots+F_{t+s-1}\le m<F_t+F_{t+1}+\cdots+F_{t+s}$. Then \begin{enumerate} \item If $s$ is odd, then $(m,m+k;r)\in\mcN$, \item If $s$ is even, then $(m,m+k;r)\in\mcP$. \end{enumerate} \item If $z_1(k)=F_{t+1}$, and $z_2(k)=F_{t+d}$, and $m\ge F_t+F_{t+1}+\cdots+F_{t+d-3}$, then \begin{enumerate} \item If $d$ is odd, then $(m,m+k;r)\in\mcN$, \item If $d$ is even, then $(m,m+k;r)\in\mcP$. \end{enumerate} \end{enumerate} \end{thm}

\begin{rem} For $s\ge 1$, the partial sums $F_t+F_{t+1}+\cdots+F_{t+s-1}$ can be written more concisely: \[ F_t+F_{t+1}+\cdots+F_{t+s-1} = (F_{t+2}-F_{t+1}) + (F_{t+3}-F_{t+2}) + \cdots + (F_{t+s+1}-F_{t+s}) = F_{t+s+1}-F_{t+1}.\] However, in this context it is more natural to leave the series unsummed, as it serves as a reminder of how the game might be played. \end{rem}

Since Theorem~\ref{thm:twopiles} is a bit complicated, let us say something about how it should be interpreted from a player's point of view. First, if it possible to remove $z_1(k)=:F_e$ stones from the $m+k$ pile, then either this move or removing $F_{e-1}$ stones from the $m$ pile is winning. (See the proof for a more complete description of when to play each of these moves.) If it is not possible to remove $z_1(k)$ stones from the $m+k$ pile, then every move \emph{in that pile} is losing. However, there may still be winning moves in the $m$ pile, and indeed the only move that \emph{might} win is to remove $F_t$ stones from the $m$ pile. If $z_1(k)\ge F_{t+2}$, then this move loses. If $z_1(k)=F_{t+1}$, then the situation is rather complicated, leading to cases (3)--(5) in the theorem. However, when actually playing the game, the structure of the theorem is not terribly important: if all moves except for one are clearly losing and the remaining one leads to complications, by all means play the complicated one!


We now turn to the proof. One key input is the following Lemma, which we used in our earlier work on the local move dynamic game:

\begin{lem}[\cite{LRS14}, Lemma 4.3] \label{lem:smallfibs} Suppose $n>1$ and $1\le k<z_1(n)$. Then $z_1(n-k)\le 2k$. \end{lem}

\begin{proof}[Proof of Theorem~\ref{thm:twopiles}] We work one case at a time. For the claimed $\mcN$ positions, we show that there is a move to a position that we claim to be in $\mcP$, and for the claimed $\mcP$ positions, we show that every move is to a claimed $\mcN$ position. By Proposition~\ref{prop:partition}, the claimed $\mcN$ and $\mcP$ positions are, in fact, the $\mcN$ and $\mcP$ positions. Observe that every position of two-pile \textsc{Fibonacci nim} is of type (1), (2), (3), (4a), (4b), (5a), or (5b).

We begin with positions of type (1). Suppose first that $z_2(k)\ge F_{t+3}$. Then we can remove $z_1(k)=:F_e$ stones from the $m+k$ pile to get to $(m,m+k-z_1(k);2z_1(k))$, which is of type (2), since $z_1(k-z_1(k))=z_2(k)\ge F_{t+3}$, which is at least as large as the second Fibonacci number after $2z_1(k)<F_{t+2}$. If $z_2(k)=F_{e+2}$ and $m<F_{e+1}$, then $(m,m+k-z_1(k);2z_1(k))$ is of type (3). 

However, if $z_2(k)=F_{e+2}$ and $m\ge F_{e+1}$, then removing $z_1(k)$ stones yields a position of type (4) or (5). Instead, there is a winning move in the $m$ pile, to $(m-F_{e-1},(m-F_{e-1})+(F_{e-1}+k);2F_{e-1})$; since $z_1(F_{e-1}+k)\ge F_{e+3}$, this position is of type (2).

Suppose we are in a position of type (2). Then we may move in the $m+k$ pile, to $(m,m+k-a;2a)$ for $1\le a\le r$, and since $r<z_1(k)$ and hence $a<z_1(k)$, Lemma~\ref{lem:smallfibs} ensures that $z_1(k-a)\le 2a$, so $(m,m+k-a;2a)$ is of type (1). We may also move in the $m$ pile to $(m-a,(m-a)+(a+k);2a)$ for $1\le a\le\min(r,m)$, which is of type (1) since $z_1(a+k)=z_1(a)\le a\le 2a$.

Now, suppose we are in a position of type (3). Then the same arguments as for type (2) positions again shows that all moves from type (3) positions are to type (1) positions.

Now, suppose we are in a position of type (4) or (5). (We will distinguish the types more finely later.) We may move in the $m+k$ pile to $(m,m+k-a;2a)$ for $1\le a\le r$, which is of type (1). We may also move in the $m$ pile to $(m-a,(m-a)+(a+k);2a)$ for $1\le a\le\min(m,r)$. If $a\neq F_t$, then $z_1(a+k)=z_1(a)\le a\le 2a$, which is of type (1). The remaining move is to $(m-F_t,(m-F_t)+(F_t+k);2F_t)$, which is of type (4) or (5) if $m-F_t\ge F_{t+1}$ and $z_1(F_t+k)=F_{t+2}$, type (2) if $z_1(F_t+k)\ge F_{t+3}$, and type (3) if $m-F_t<F_{t+1}$ and $z_1(F_t+k) = F_{t+2}$. As a result, the only move from a position of type (4) or (5) that \emph{might} be to a $\mcP$ position is the move to $(m-F_t,(m-F_t)+(F_t+k);2F_t)$, and only if this position is of type (4) or (5). When it is to another position of type (4) or (5), then it decreases $s$ or $d$ by one, depending on whether it is a type (4) or (5) position, respectively. Hence, a position $(m,m+k;r)$ of type (4) or (5) is a $\mcP$ position if the (unique) maximal sequence of moves to positions of type (4) or (5) has even length, and is an $\mcN$ position if the (unique) maximal sequence of moves to positions of type (4) or (5) has odd length. From a position of type (4), removing consecutive Fibonacci numbers from the $m$ pile eventually results in a position of type (3), whereas from a position of type (5), removing consecutive Fibonacci numbers from the $m$ pile eventually results in a position of type (2). Either way, this distinguishes types (4a) and (4b), as well as (5a) and (5b). \end{proof}

\subsection{The word approach}
As in the case of one-pile \textsc{Fibonacci nim}, it is possible to express the $\mcP$-positions of two-pile global \textsc{Fibonacci nim} in terms of partial sums of a word. 

Fix the smallest pile size $m\ge 0$ in a two pile game, and define $p=p(m)\ge 0$ as a function of $m$ such that $F_p\le m<F_{p+1}$. If $r < F_{p-1}$, then we say that the position $(m,m+k;r)$ is \emph{hybrid}, and otherwise it is \emph{Sturm} (in particular, the latter case applies if $p>0$ and  $r\ge F_{p-1}$). We also define any one-pile game to be Sturm. 
Let $\alpha =\alpha(r,p)$ be the function of $r$ and $p$, defined by $F_{p+\alpha-1}\le r < F_{p+\alpha}$ (so the parameter $t$ from Theorem~\ref{thm:alg} is $t=p+\alpha-1$). For $\alpha\ge 0$, this function will classify the Sturm games that are $\mcP$-positions (via the word $w_{p+\alpha}$ or $w_{p+\alpha-1}$, as defined in Lemma~\ref{lem:partialsums}).

For the hybrid games, we recursively build the relevant word for classifying the $\mcP$-positions, in the following way. There are three possible transformations of a letter in a given word. In each word, each letter is one of three consecutive Fibonacci numbers, generalizing the two letter Sturm case. The Fibonacci numbers (as letters) are given recursively by successively decreasing $\alpha$ via the move dynamic parameter $r$, starting with the Sturm case of $\alpha = 0$. 

The possible transformations are
\begin{enumerate}
\item[(T1)] $F_i\rightarrow F_{i}$; ($y\rightarrow y$)
\item[(T2)] $F_i\rightarrow F_{i-1}F_{i-2}$; ($x\rightarrow yz$, or $y\rightarrow vz$)
\item[(T3)] $F_i\rightarrow F_{i-2}F_{i-3}F_{i-2}$; ($x\rightarrow vzv$)
\end{enumerate}
The only new transformation is (T3), and it applies only if $x$ is the largest Fibonacci number present in the word, (that is only if the alphabet for the current word is $\{F_i,F_{i-1},F_{i-2}\}$). (We also list an interpretation in symbolic dynamics, generalizing Proposition~\ref{prop:fibword} (3), for later reference. For example starting with the Fibonacci word on the alphabet $\{x,y\}$, we produce a ``generalized Fibonacci word'' on the alphabet $\{y,z,v\}$, by applying (T3) and (T1): $vzvyvzvvzvyvzvyvzv\ldots$. For each transformation to follow, at most one of the two possible transformations in (T2) is used. The symbolic dynamic approach will give an abstract interpretation of the transformation $T$ below, and could be studied independently.\footnote{Note the similarity with the Tribonacci morphism $x\rightarrow xy,y\rightarrow xz,z\rightarrow x$, which satisfies $T_n=T_{n-1}T_{n-2}T_{n-3}$, with fixed point $xyxzxyxxyxzxyx\underline{y}xz\ldots$. We have indicated the first letter that differs from our type 1 word in Example~\ref{ex1}.}) Note that (T3) is the second iteration of the Fibonacci morphism, that is, the word $S_2$ in Proposition~\ref{prop:fibword}. 

Before introducing the transformation $T$, let us give three generic examples of how (T1), (T2), (T3) apply. 

\begin{ex}[Type 1] The word transformation \label{ex1} $$[34, 21, 34, 34, 21, 34, 21, 34, \ldots] \rightarrow [13, 8, 13, 21, 13, 8, 13, 13, 8, 13, 21, 13, 8, 13, 21, 13, 8, 13,\ldots ],$$ appears when $m=26$ and the move dynamic parameter changes from $r=13$ to $r=12$, i.e.\ when $\alpha$ changes from 0 to $-1$. This means that we start with the Fibonacci word on the alphabet $\{34,21\}$ and apply (T1) to 21 and (T3) to 34. Note that the decrease in $r$ gives a new $\alpha$, and hence a new word.
\end{ex}

The case (T2) occurs only in very particular cases, concerning the two largest letters (the smallest letter uses only (T1)).

\begin{ex}[Type 2] \label{ex2} The word transformation $$[13, 8, 13, 8, 5, 8, 13, 8, 13, 13, 8, \ldots]\rightarrow [8, 5, 8, 5, 3, 5, 8, 5, 8, 8, 5, 8, 5, 3, 5, 8, 5, 8, \ldots],$$ occurs
when $m=26$ and $r$ changes from 5 to 4, so that $\alpha$ changes from $-2$ to $-3$. In this example, ``13'' changes to ``8,5'' (T2) if the partial sum of all lower terms belongs to $PS(w_{9})$ (or equivalently $PS(w_8)$, since we are only concerned with the letter ``13," which is the word at level $\alpha =0$ in Example~\ref{ex1}), and otherwise ``13'' changes to ``5,3,5'' (T3). For reference to the definition to come, here $x=34-26=8=F_6$, and $p=8$. \end{ex}

\begin{ex}[Type 3] \label{ex3} The word transformation $$[13, 8, 13, 21, 13, 8, 13, 13, 8, 13\ldots]\rightarrow [8, 5, 8, 13, 8, 5, 8, 8, 5, 8, 13, 8, 5, 8, 13\ldots],$$ occurs when $m=25$ and $r$ changes from 8 to 7, so that $\alpha$ changes from $-1$ to $-2$. In this example, ``13'' changes to ``8,5'' (T2) if the (partial) sum of all lower terms belongs to $PS(w_{9})$ (or $PS(w_{8})$), 
and otherwise it does not decompose (T1). For use in the definition to come, here $x=34-25=9\le F_7$, and $p=8$. \end{ex}

Let us describe the words for consideration. For $\alpha \ge 0$, we let $T^{-\alpha}(w_p)$ denote the word where the reverse of (T1) or (T2) has been applied to each letter in the word $T^{1-\alpha}(w_p)$, as follows. Suppose first that $\alpha = 0$, in which case the word is $T^0(w_p) = w_p = W^{F_{p+1},F_p}$, which is defined as in the one pile case. The Fibonacci morphism applies (with only one important exception) that is (T2) is applied to each instance of the larger letter and (T1) to each instance of a smaller letter (as in the one pile case). The exception is the case $\alpha = -1$, for which $T^{-1}(w_p) = T^{-2}(w_p)$, where only (the reverse of ) (T1) is used. Thus, when $\alpha\ge 0$, each word $T^{-\alpha}(w_p)$ is the Fibonacci word, and only the alphabet differs.

In a transformation to the hybrid case, i.e. $\alpha \le 0$, except for two \emph{special} cases (to be described in the next paragraph), the transformation $T$ applies (T3) to the largest Fibonacci letter in the current word, and (T1) to the other letter(s). For the special cases: define $x$ by $m = F_{p+1} - x\ge F_p$, 

A \emph{special transformation} will apply if and only if $F_{p+\alpha} < x \le F_{p+\alpha +1}$ (where $\alpha$ is defined by $F_{p+\alpha -1}\le r < F_{p+\alpha}$ as usual). There are two cases for this special transformation, depending on whether $p+\alpha$ is even or odd. Denote the partial sum of a word $W$ of all terms with index less than the $i^\text{th}$ letter by $W(i)$. 
The current word (before the transformation) is $T^{1-\alpha}(w_p)$. 

\begin{itemize}

\item Suppose first that $p+\alpha$ is odd. Consider the $i^\text{th}$ letter if and only if it is the largest letter. Then (Example~\ref{ex2} generalizes to) (T2) applies if $T^{1-\alpha}(w_p)(i)\in PS(w_p)$, and otherwise apply (T3). For the second largest and the smallest letter (T1) applies.

\item Suppose next that $p+\alpha$ is even. Consider the $i^\text{th}$ letter if and only if it is the second largest letter. Then (Example~\ref{ex3} generalizes to) (T2) applies if $T^{1-\alpha}(w_p)(i)\in PS(w_p)$, and otherwise (T1); (T3) applies to the largest letter and (T1) to the smallest letter.

\end{itemize}

Note that Example~\ref{ex1} is $T(w_p)$ for $p=8$ (which is non-Sturmian although $w_p$ is Sturmian). This type 1 transformation also applies to each purely hybrid non-special transformation.

In the case of a Sturm position $(m,m+k;r)$ (with $F_p\le m<F_{p+1}$), let $\sigma_m(\alpha)=PS(w_{p+\alpha})$ if $0\le \alpha \le 1$ and $\sigma_m(\alpha) = PS(w_{p+\alpha-1})$ if $\alpha > 1$, and otherwise, in case of a hybrid position, that is, if $\alpha<0$, we let $\sigma_m(\alpha) = PS(T^{-\alpha}(w_p))$. By this maneuver, we combine the hybrid and Sturm notations, and each set of partial sums of a word will be some $\sigma_m(\alpha)$ for some integer $\alpha = \alpha(r,p(m))$. Thus, for a fixed $m$, and variable $r$, $\sigma_p(0)$ represents the first Sturm set $PS(w_p)$, $\sigma_m(1)=\sigma_m(2)$ the second, $\sigma_m(3)$ the third, and so forth. $\sigma_m(-1)$ is the first hybrid set (Example~\ref{ex1}), and, the interesting special cases (Examples~\ref{ex2} and~\ref{ex3}) appear as $\sigma_m(-2)$, because $F_{p+\alpha-1} < F_{p+1}-m=8 \le F_{8-2} = F_{p+\alpha}$, and $\sigma_m(-1)$, because $F_{p+\alpha-1} < F_{p+1}-m=9 \le F_{8-1} = F_{p+\alpha}$, respectively. Note that, in the hybrid case, $\sigma_m(\alpha)$ depends on $x=F_{p+1}-m$, namely in case of $r<F_\xi$, where $\xi$ is defined by $F_\xi < x \le F_{\xi+1}$. Hence, in general we cannot only rely on the (more convenient) parameter $p$.


The smallest letter in the alphabet of a given word has a similar importance to the proof of Theorem~\ref{thm:alg2} as for that of Theorem~\ref{thm:alg}. 
\begin{lem}
Consider the hybrid case with $\alpha > -p+2$. If $\alpha = -p +3<0$, then the smallest letter is $F_2=1$; that is, the smallest letter in $T^{p-3}(w_p)$, $p>3$ is $F_2$. For $\alpha < 0$, the smallest letter in $T^{p-3-j}(w_p)$ is $F_{2+j}$. That is, for $\alpha <0$, the smallest letter in $T^{-\alpha}(w_p)$ is $F_{p + \alpha - 1}$. For the Sturm case, if $0\le \alpha\le 1$, then the smallest letter is $F_{p + \alpha }$, and otherwise it is $F_{p + \alpha -1}$. The largest letter is two larger than the smallest in case of hybrid and one larger than the smallest in case of Sturm.
\end{lem}
\begin{proof}
This follows from the definition of the word $T^\alpha(w_p)$.
\end{proof}
That is, the smallest letter is $F_{p + \alpha - 1}$, except for the cases $0\le \alpha \le 1$, when it is $F_{p+\alpha}$ (independently of Sturm or hybrid). 

\begin{thm}\label{thm:alg2}
The position $(m,m+k;r)\in\mcP$ if and only if $k\in \sigma_m(\alpha)$.
\end{thm}

The proof is similar to that of Theorem~\ref{thm:alg}. There are more tedious details to check, but all the ideas and techniques are similar.

\section{Multi-pile \textsc{Fibonacci nim}} \label{sec:multipile}

The following theorem about (ordinary) \textsc{nim} is well-known:

\begin{thm}[Bouton] If $(n_1,\ldots,n_k)$ is a \textsc{nim} position, then there is a unique nonnegative integer $b$ for which $(n_1,\ldots,n_k,b)\in\mcP$. Furthermore, $b<2\max(n_1,\ldots,n_k)$ and $b\le n_1+\cdots+n_k$. \end{thm}

\begin{rem} This follows easily from Bouton's description of the winning strategy in \textsc{nim}, but it is also possible to prove it (say, by induction on the largest power of 2 occurring in any $n_i$) without using the winning strategy. \end{rem}

It would be desirable to have a similar statement for multi-pile \textsc{Fibonacci nim}. However, the best we can do is the following:

\begin{thm} \label{thm:comps} If $(n_1,\ldots,n_k;\infty)$ is a multi-pile \textsc{Fibonacci nim} position, then there is at most one nonnegative integer $b$ for which $(n_1,\ldots,n_k,b;\infty)\in\mcP$. When it exists, we call $b$ the \emph{complementary value} of $(n_1,\ldots,n_k)$. \end{thm}

\begin{proof} Suppose that $(n_1,\ldots,n_k,b;\infty)\in\mcP$, and suppose that $b'>b$. Then there is a move from $(n_1,\ldots,n_k,b';\infty)$ to $(n_1,\ldots,n_k,b;b'-b)$. But the latter is a $\mcP$ position, since its options are a subset of those of $(n_1,\ldots,n_k,b;\infty)$, which is itself a $\mcP$ position. Hence $(n_1,\ldots,n_k,b';\infty)\in\mcN$. \end{proof}

\begin{rem} It remains an open question to determine a bound on the complementary value, when it exists. It appears that most of the time, the complementary value is ``not too much larger'' than the maximum of the $n_i$'s. However, there are some notable exceptions; for instance: \begin{itemize} \item $(1,47,72;\infty)\in\mcP$, \item $(2,41,139;\infty)\in\mcP$, \item $(2,93,345;\infty)\in\mcP$, \item $(8,9,53;\infty)\in\mcP$. \end{itemize} See Table~\ref{table:comps} for a table of values of complementary values. \end{rem}

\begin{table} \begin{tabular}{c|cccccccccccccccc} & 0 & 1 & 2 & 3 & 4 & 5 & 6 & 7 & 8 & 9 & 10 & 11 & 12 & 13 & 14 & 15 \\ \hline 0 & 0 & 1 & 2 & 3 & 4 & 5 & 6 & 7 & 8 & 9 & 10 & 11 & 12 & 13 & 14 & 15 \\ 1 & 1 & 0 & 4 & 6 & 2 & 9 & 3 & 11 & 12 & 5 & 14 & 7 & 8 & 17 & 10 & 16 \\ 2 & 2 & 4 & 0 & 7 & 1 & 10 & 11 & 3 & 19 & 15 & 5 & 6 & 14 & 24 & 12 & 9 \\ 3 & 3 & 6 & 7 & 0 & \fbox{$\infty$} & 11 & 1 & 2 & 16 & 12 & 13 & 5 & 9 & 10 & 17 & 18 \\ 4 & 4 & 2 & 1 & \fbox{$\infty$} & 0 & 7 & 10 & 5 & 17 & 16 & 6 & 18 & 13 & 12 & 19 & 69 \\ 5 & 5 & 9 & 10 & 11 & 7 & 0 & 35 & 4 & 15 & 1 & 2 & 3 & 18 & 22 & 23 & 8 \\ 6 & 6 & 3 & 11 & 1 & 10 & 35 & 0 & 8 & 7 & 17 & 4 & 2 & 16 & 14 & 13 & 26 \\ 7 & 7 & 11 & 3 & 2 & 5 & 4 & 8 & 0 & 6 & 13 & 27 & 1 & 15 & 9 & 22 & 12 \\ 8 & 8 & 12 & 19 & 16 & 17 & 15 & 7 & 6 & 0 & 53 & 11 & 10 & 1 & 57 & 35 & 5 \\ 9 & 9 & 5 & 15 & 12 & 16 & 1 & 17 & 13 & 53 & 0 & 21 & 27 & 3 & 7 & 76 & 2 \\ 10 & 10 & 14 & 5 & 13 & 6 & 2 & 4 & 27 & 11 & 21 & 0 & 8 & 26 & 3 & 1 & 24 \\ 11 & 11 & 7 & 6 & 5 & 18 & 3 & 2 & 1 & 10 & 27 & 8 & 0 & 22 & 21 & 64 & 88 \\ 12 & 12 & 8 & 14 & 9 & 13 & 18 & 16 & 15 & 1 & 3 & 26 & 22 & 0 & 4 & 2 & 7 \\ 13 & 13 & 17 & 24 & 10 & 12 & 22 & 14 & 9 & 57 & 7 & 3 & 21 & 4 & 0 & 6 & 20 \\ 14 & 14 & 10 & 12 & 17 & 19 & 23 & 13 & 22 & 35 & 76 & 1 & 64 & 2 & 6 & 0 & 21 \\ 15 & 15 & 16 & 9 & 18 & 69 & 8 & 26 & 12 & 5 & 2 & 24 & 88 & 7 & 20 & 21 & 0 \end{tabular} \caption{Complementary values of \textsc{Fibonacci nim}. The boxed $\infty$'s mean that there is no complementary value for these positions.} \label{table:comps} \end{table}

A curious aspect of Theorem~\ref{thm:comps} is the possibility that there may not be a complementary value for a \textsc{Fibonacci nim} position. It turns out that \textsc{Fibonacci nim} positions with no complementary values \emph{do} exist.

\begin{thm} For any nonnegative integer $n$, $(3,4,n;\infty)\in\mcN$. \end{thm}


It turns out that many small facts have to be verified in this proof. Since it is tedious to check the details, we only explain the general ideas.

\begin{proof} The following four classes partition the nonnegative integers. 
\begin{enumerate}
\item $B-2=PS(w_3)=\{n:z_1(n)\ge 3\}=\{0,3,5,8,11,13,16,18,21,\ldots\}$,
\item $AB-2=1+PS(w_4)=\{n:z_1(n-1)\ge 5\}=\{1,6,9,14,19,22,27,\ldots\}$,
\item $AB-1=2+PS(w_4)=\{n:z_1(n-2)\ge 5\}=\{2,7,10,15,20,23,28,\ldots\}$,
\item $BB-1=4+PS(w_5)=\{n:z_1(n-4)\ge 8\}=\{4,12,17,25,33,38,\ldots\}$.
\end{enumerate}
(Recall that $F_3=2$, so that $PS(w_3)$ consists of partial sums with letters 3 and 2, and so forth.)


\begin{rem} The names for these sets come from the theory of complementary equations. We let $a(n)=\lfloor\phi n\rfloor$ and $b(n)=\lfloor \phi^2 n\rfloor$, where $\phi=\frac{1+\sqrt{5}}{2}$. Then $A$ consists of all numbers of the form $a(n)$ for some $n\ge 1$, $B$ consists of all numbers of the form $b(n)$ for some $n$, $AB$ consists of all numbers of the form $a(b(n))$ for some $n$, and so forth. See~\cite{Kim08} for more details. \end{rem}

Adding one to each set gives the sets, $B-1 = AA$, $AB-1 = BA$, $AB$ and $BB$. The sets $AA$ and $AB$ partition $A$ and the sets $BA$ and $BB$ partition $B$. $A$ and $B$ partition the positive integers.


We claim that the following moves are to $\mcP$ positions:
\begin{enumerate}
\item If $n\in B-2$, then $(3,3,n;2)\in\mcP$.
\item If $n\in AB-2$, then $(2,4,n;2)\in\mcP$.
\item If $n\in AB-1$, then $(1,4,n;4)\in\mcP$.
\item If $n\in BB-1$, then $(0,4,n;6)\in\mcP$.
\end{enumerate}

We give a proof in the spirit of the proofs of Theorems~\ref{thm:alg} and~\ref{thm:alg2}, although it is also possible to give a proof in terms of the Zeckendorf representation. 

Part (4) of the claim follows from Theorem~\ref{thm:alg2}. Let us list the moves for the respective three first types (enumerating as above, and with $n$ belonging to respective subclass):
\begin{enumerate}
\item $(3,3,n-x;2x), 1\le x\le 2$, $(2,3,n;2)$, $(1,3,n;4)$ (we may assume $n>0$).
\item $(2,4,n-x;2x), 1\le x\le 2$, $(2,3,n;2)$, $(2,2,n;4)$, $(1,4,n;2)$, $(0,4,n;4)$.
\item $(1,4,n-x;2x), 1\le x\le 4$, $(0,4,n;2)$, $(1,3,n;2)$, $(1,2,n;4)$, $(1,1,n;6)$, $(0,1,n;8)$.
\end{enumerate}
We must show that all of these positions are in $\mcN$. To do this, we show that all of the following positions are in $\mcP$:
 \begin{itemize}
\item $(0,1,z;2)\in \mcP$ if $z\in B-1$, and $(0,1,z;6)\in \mcP$ if $z\in BB-3\subset AB-2\subset B-1$.
\item $(0,2,z;4)\in \mcP$ if $z\in AB-1$.
\item $(0,3,z;2)\in \mcP$ if $z\in AB=\{3,8,11,16,21, 24\ldots\}\subset B-2$.
 \item $(1,1,z;2)\in \mcP$ if $z\in B-2$ and $(1,1,z;4)\in \mcP$ if $z\in BB\subset B-2$. 
 \item $(1,2,z; 2)\in \mcP$ if $z\in BB-1$.
 \item $(1,3,z; 2)\in \mcP$ if $z\in AB-2$.
\item $(2,2,z;2)\in \mcP$ if $z\in B-2$ and $(2,2,z;4)\in \mcP$ if $z\in BB\subset B-2$.
 \item $(2,3,z;2)\in \mcP$ if $z\in AB-1$.
 \end{itemize}

We verify that each candidate $\mcN$ position in (1) above has a move to a candidate $\mcP$ position. Since $2<n\in B-2$, the positions of the form $(3,3,n-x;2x)$ can immediately be reverted to a position of the same type. This follows from the proof of Theorem~\ref{thm:alg}, using the letters $(3,2)$. From $(2,3,n;2)$ we can move to the candidate $\mcP$ position $(2,2,n;2)$, and from $(1,3,n;4)$ we can move either to $(1,3,n-2;4)$ (with $n\in B-2\setminus AB=\{5,13,18,\ldots\}$, which implies $n-2\in B-2$), or to $(0,3,n;2)$ (with $n\in AB$). 
 
Next, we verify that each candidate $\mcN$ position in (2) above has a move to a candidate $\mcP$ position. Here $0 < n\in AB-2$ and the letters are $(5,3)$. For positions of the form $(2,4,n-x;2x)$, all but one case can be reversed to a $\mcP$ position of the same type. This is the case where the current letter of the Fibonacci word with letters $3,5$ is 5, but $x=1$. In this case, $n-1\in BB$, so that $(2,2,n-1;4)$ is a $\mcP$ candidate. (In fact, we have that $BB\subset AB-3\subset B-2$.) For the remaining three proper three-pile cases, we have just seen that $(2,2,n-1;4)$ is a $\mcP$ candidate; it remains to note that both $(2,3,n;2)$ and $(1,4,n;2)$ have options to the candidate $\mcP$ position $(1,3,n;2)$.

For case (3), concerning $n\in AB-1$, the first type has a reversible option unless $x=1$ and the current letter is 5. In any case, the option $(1,3,n-1;2)$ is a candidate $\mcP$ position, and this also suffices for $(1,3,n;4)$. There are two more proper 3-pile candidate $\mcN$ positions of this form: $(1,2,n;4)$ reverses to $(0,2,n-1;2)$ which is a $\mcP$ candidate, and from $(1,1,n;6)$, there is an option $(1,1,z;2(n-z))$, with $z \in B-2$, because the move dynamic is 6. Hence, we have found $\mcP$ candidates for all the $\mcN$ candidates of the forms in (1)--(3).

Next, we must show that for all the lower level $\mcP$ candidates, for all $\mcN$ candidate options, there is a reversible move to an $\mcP$ candidate. Let us begin to show that $(1,1,z;2)$ is reversible. The two-stone removal is reversible, by the above argument, and the move to $(1,1,z-1;2)$ reverts to a position of the same form, unless the current letter is ``3.'' In this case there is a response  to $(0,1,z-2;2)$, and where $z-2\in AB-2$, which is a $\mcP$ position, by Theorem 13. This response is also possible from $(0,1,z;2)$, which concludes this case. 

For a candidate $\mcP$ position of the form $(1,2,z;2)$, note that $BB-3\subset AB-1$, and so both $(0,2,z;2)$ and $(1,2,z-2;4)$ revert to the $\mcP$ position $(0,2,z-2;r)$, $r=2,4$ respectively. A move to $(1,1,z;2)$ reverts to the $\mcP$ position $(1,1,z-1;2)$, because $z-1\in B-2$. A move to $(0,1,z;4)$ reverses to $(0,1,z-3;6)$, which is a $\mcP$ position (by Theorem~\ref{thm:alg2}). The move to $(1,2,z-1;2)$, reverses to $(1,1,z-1;2)$, which is a $\mcP$ position because $z-1\in B-2$. 

For the position $(1,3,z; 2)$, with $z\in AB-2$, if two stones are removed from the third pile, the position reverses to one of the same form. Similarly, from $(1,3,z-1; 2)$, it suffices to study the ``5" letter case, and thus $z-1\in BB$; there is a response to $(1,1,z-1; 4)\in \mcP$. Next, consider $(1,2,z; 2)$, with $z\in AB-2$; then respond to $(0,1,z; 4)\in \mcP$. Consider $(1,1,z; 4)$, with $z\in AB-2$; then respond to $(0,1,z; 2)\in \mcP$. The options $(0,1,z; 6)$ and $(0,3,z; 2)$, with $z\in AB-2$, are both $\mcN$ positions, by Theorem~\ref{thm:alg2}. 

For the position $(2,2,z; r)$, with $z\in B-2$ and $r=2,4$, playing on the third pile is reversible to a position of the same type. 
Playing on the first pile,  $(1,2,z; 2)$ reverses to $(1,1,z; 2)$ and playing to $(0,2,z; 4)$, gives an $\mcN$ position, by Theorem~\ref{thm:alg2}.

For the position $(2,3,z; 2)$, with $z\in AB-1$, playing on the third pile, it suffices to find a winning response to the option $(2,3,z-1; 2)$ when the current letter is a ``5,'' and therefore with $z-1\in BB-3\subset AB-2$. The option $(1,3,z-1; 2)\in \mcP$  suffices (so the letter ``5'' is not important). The same response obviously works for the option $(1,3,z; 2)$. The remaining options to check are $(0,3,z; 4)$ (which is an $\mcN$ position by Theorem~\ref{thm:alg2}), $(2,2,z; 2)$, and $(1,2,z; 4)$. These options have responses to $(0,2,z; r)\in \mcP$, for $r=2,4$. 
\end{proof}

\begin{qn} Are there any other two-pile \textsc{Fibonacci nim} positions $(n_1,n_2;\infty)$ (besides $(3,4;\infty)$) with no complementary value? \end{qn}

\section{An easier variant: global \textsc{power-of-two nim}} \label{sec:pow2nim}

\textsc{power-of-two nim} is a simpler variant of \textsc{Fibonacci nim}. In the classical (one-pile) formulation, the rules are the same as in \textsc{Fibonacci nim}, except that if the previous player removed $m$ stones, then the next player may only remove at most $m$ stones. Thus, the move dynamic can only stay the same or decrease on each move.

The winning strategy is closely related to that of \textsc{Fibonacci nim}, but it relies on the binary representation of $n$ rather than the Zeckendorf representation. A winning strategy is to remove the smallest bit from the binary representation of the pile size on each move, and a position is a $\mcP$ position if and only if the smallest bit is larger than the move dynamic.

We represent the multi-pile global \textsc{power-of-two nim} game using the same notation as we do the multi-pile global \textsc{Fibonacci nim} game. It turns out that we can describe the $\mcP$ positions of multi-pile \textsc{power-of-two nim} completely. Let $a\oplus b$ denote the nim sum of $a$ and $b$, and let $\sb(n)$ denote the smallest power of 2 in the binary expansion of $n$, i.e.\ if $n=2^{a_1}+\cdots+2^{a_k}$ where the $a_i$'s are distinct powers of 2 with $a_1<\cdots<a_k$, then $\sb(n)=2^{a_1}$. (If $n=0$, then we define $\sb(n)=\infty$.) Then we have the following:

\begin{thm} \label{thm:powtwosoln} The \textsc{power-of-two nim} game $(n_1,\ldots,n_k;r)\in\mcP$ if and only if $\sb(n_1\oplus\cdots\oplus n_k)>r$. \end{thm}

\begin{cor} \label{cor:powtwocor} The \textsc{power-of-two nim} game $(n_1,\ldots,n_k;\infty)\in\mcP$ if and only if the \textsc{nim} game $(n_1,\ldots,n_k)\in\mcP$. \end{cor}




\begin{proof}[Proof of Theorem~\ref{thm:powtwosoln}] The idea is to mimic good play in \textsc{nim}, playing a move that makes partial progress toward a winning \textsc{nim} move. To this end, we show that, given a position that we claim to be an $\mcN$ position, there is a move to a position that we claim to be a $\mcP$ position, whereas given a claimed $\mcP$ position, all moves are to claimed $\mcN$ positions. By Proposition~\ref{prop:partition}, this shows that the $\mcP$ positions are exactly as we claim them to be.

First, suppose that $(n_1,\ldots,n_k;r)$ is a \textsc{power-of-two nim} position with $\sb(n_1\oplus\cdots\oplus n_k)\le r$. Then there is some move $n_i\to n_i'$ that is a winning move in \textsc{nim}. Let $2^a=\sb(n_1\oplus\cdots\oplus n_k)$. Then $2^a\le r$, so removing $2^a$ stones from pile $i$ is a legal move, to $(n_1,\ldots,n_i',\ldots,n_k;2^a)$. But now $\sb(n_1\oplus\cdots\oplus n_i'\oplus\cdots\oplus n_k)\ge 2^{a+1}$, so this position is a claimed $\mcP$ position.

On the other hand, suppose that $\sb(n_1\oplus\cdots\oplus n_k)>r$, and consider the move $n_i\to n_i'$, where $n_i-n_i'\le r$. Then $\sb(n_1\oplus\cdots n_i'\oplus\cdots\oplus n_i)=\sb(n_i-n_i')$, so the position $(n_1,\ldots,n_i',\ldots,n_k;n_i-n_i')$ is a claimed $\mcN$ position. This completes the proof of Theorem~\ref{thm:powtwosoln}. \end{proof}

\bibliography{globalfibnim}
\bibliographystyle{halpha}

\end{document}